\documentclass[letterpaper, 10 pt, conference]{ieeeconf}
\IEEEoverridecommandlockouts
\usepackage{amsmath,amssymb,nicefrac,bm,upgreek,mathtools,verbatim,enumerate,cite}
\usepackage[final]{hyperref}
\usepackage[mathscr]{eucal}
\usepackage{algorithmic}
\usepackage{graphicx}
\usepackage{textcomp}
\usepackage{latexsym}

\usepackage{xcolor}
\definecolor{mred}{rgb}{.5,.0,.0}
\definecolor{dmagenta}{rgb}{.4,.1,.5}
\definecolor{dblue}{rgb}{.0,.0,.5}
\definecolor{mblue}{rgb}{.0,.0,.7}
\definecolor{ddblue}{rgb}{.0,.0,.4}
\definecolor{dred}{rgb}{.7,.0,.0}
\definecolor{dgreen}{rgb}{.0,.5,.0}
\definecolor{Eeom}{rgb}{.0,.0,.5}

\def\BibTeX{{\rm B\kern-.05em{\sc i\kern-.025em b}\kern-.08em
    T\kern-.1667em\lower.7ex\hbox{E}\kern-.125emX}}

\hypersetup{
  colorlinks=true,
  citecolor=mblue,
  linkcolor=mblue,
  frenchlinks=false,
  pdfborder={0 0 0},
  naturalnames=false,
  hypertexnames=false,
  breaklinks}
\usepackage[capitalize,nameinlink]{cleveref}

\crefname{section}{Section}{Sections}
\crefname{subsection}{subsection}{subsections}
\crefname{notation}{Notation}{Notations}
\crefname{hypothesis}{Hypothesis}{Conditions}
\crefname{assumption}{Assumption}{Assumptions}

\Crefname{figure}{Figure}{Figures}

\crefformat{equation}{\textup{#2(#1)#3}}
\crefrangeformat{equation}{\textup{#3(#1)#4--#5(#2)#6}}
\crefmultiformat{equation}{\textup{#2(#1)#3}}{ and \textup{#2(#1)#3}}
{, \textup{#2(#1)#3}}{, and \textup{#2(#1)#3}}
\crefrangemultiformat{equation}{\textup{#3(#1)#4--#5(#2)#6}}%
{ and \textup{#3(#1)#4--#5(#2)#6}}{, \textup{#3(#1)#4--#5(#2)#6}}%
{, and \textup{#3(#1)#4--#5(#2)#6}}

\Crefformat{equation}{#2Equation~\textup{(#1)}#3}
\Crefrangeformat{equation}{Equations~\textup{#3(#1)#4--#5(#2)#6}}
\Crefmultiformat{equation}{Equations~\textup{#2(#1)#3}}{ and \textup{#2(#1)#3}}
{, \textup{#2(#1)#3}}{, and \textup{#2(#1)#3}}
\Crefrangemultiformat{equation}{Equations~\textup{#3(#1)#4--#5(#2)#6}}%
{ and \textup{#3(#1)#4--#5(#2)#6}}{, \textup{#3(#1)#4--#5(#2)#6}}%
{, and \textup{#3(#1)#4--#5(#2)#6}}

\crefdefaultlabelformat{#2\textup{#1}#3}
\newtheorem{theorem}{Theorem}[section]
\newtheorem{lemma}{Lemma}[section]

\newtheorem{definition}{Definition}[section]

\newtheorem{remark}{Remark}[section]
\newtheorem{example}{Example}[section]

\newcommand{\E}{\mathrm{e}}    
\newcommand{\PA}{\mathcal{P}}
\newcommand{\sG}{\mathscr{G}}

\newcommand{\cA}{\mathcal{A}}
\newcommand{\cI}{\mathcal{I}}
\newcommand{\Act}{\mathcal{U}}
\newcommand{\cM}{\mathcal{M}}
\newcommand{\cQ}{\mathcal{Q}}
\newcommand{\Usm}{\mathfrak{U}_{\mathsf{sm}}}
\newcommand{\Upm}{\mathfrak{U}_{\mathsf{p}}}
\newcommand{\smid}{\,|\,}

\newcommand{\RR}{\mathbb{R}} 
\newcommand{\NN}{\mathbb{N}} 
\newcommand{\Exp}{\mathbb{E}} 

\newcommand{\df}{\coloneqq} 

\title{\LARGE \bf
Linear and Dynamic Programs for Risk-Sensitive\\ Cost Minimization}

\author{Ari~Arapostathis$^{1}$,~\IEEEmembership{Fellow,~IEEE,}
and
Vivek S.~Borkar$^{2}$,~\IEEEmembership{Fellow,~IEEE}
\thanks{$^{1}$A. Arapostathis is with the Electrical and Computer
Engineering Department, 2501 Speedway, EER 7.824,
The University of Texas at Austin, Austin, TX 78712
(e-mail: ari@utexas.edu).}
\thanks{$^{2}$Vivek S.~Borkar is with the Department of Electrical Engineering,
Indian Institute of Technology Bombay,
Powai, Mumbai 400076, India
(e-mail: borkar.vs@gmail.com).}}%

\begin{document}

\maketitle
\thispagestyle{empty}
\begin{abstract}
We derive equivalent linear and dynamic programs for infinite horizon
risk-sensitive  control for minimization of the asymptotic growth rate of
the cumulative cost.
\end{abstract}


%
\IEEEpeerreviewmaketitle

%

\maketitle

\section{Introduction}

Risk-sensitive control problems that seek to minimize over an infinite time
horizon the asymptotic growth rate of mean exponentiated cumulative cost of
a controlled Markov chain were first studied in \cite{Flem1,Flem2}, which also
pioneered the most popular approach to such problems, viz., to use the celebrated
`log-transformation' to convert it to a zero sum stochastic game with long run
average or `ergodic' payoffs. An equivalent alternative approach that treats the corresponding reward maximization problem as a nonlinear eigenvalue problem was
developed in \cite{Ananth}. This leads to an equivalent ergodic reward maximization
problem and an associated linear program. For the finite state-action case,
the complete details of the latter were worked out in \cite{Borkar}.
Unfortunately the techniques therein do not extend to the cost minimization problem,
which is equivalent to a zero sum ergodic stochastic game.
It may be recalled that unlike the classical criteria such as discounted or ergodic,
risk-sensitive reward maximization cannot be converted to a cost minimization and vice versa,
by a simple sign flip. Thus the two are not equivalent.

In this work, we make the key observation that the aforementioned zero sum
ergodic game belongs to a very special subclass thereof, viz., a single controller
game wherein one agent affects only the payoff and not the dynamics.
This case is indeed amenable to a linear programming formulation as pointed out
in \cite{Vrieze}.
We exploit this fact to derive the counterparts of the results of \cite{Borkar}
for the cost minimization problem.
It may be noted that an LP formulation for risk-sensitive 
cost or reward is not a priori obvious because unlike the classical criteria such
as discounted or ergodic, where the uncontrolled problems lead to linear `one step analysis'
(or the Poisson equation),
risk-sensitive control leads to an eigenvalue problem which is already nonlinear.

We introduce the control problem in \cref{S2}. The equivalent single
controller ergodic game and its linear programming formulation is given in \cref{S3}.
\Cref{S4} uses this in turn to derive the corresponding dynamic programming equations
for risk-sensitive control without the assumption of irreducibility.
This leads to a second `dynamic programming' equation coupled to the usual one in
what is a counterpart of the corresponding system of equations for
ergodic control without irreducibility (\cite{Puterman}, Chapter 9).
The interesting twist here is the appearance of the so called `twisted' transition kernel.

\section{Risk-sensitive cost minimization}\label{S2}

Consider a controlled Markov chain $\{X_n\}$ on a finite state space
$S \df \{1,2,\dotsc,s\}$, controlled by a control process $\{Z_n\}$ taking values
in a finite action space $\Act$, with running cost $c(i,u)$, $i \in S$, $u \in \Act$.
Let
\begin{equation*}
(i,u,j) \in S\times \Act\times S \,\mapsto\, p(j\,|\,i,u) \in [0,1]\,,
\end{equation*}
with $\sum_j p(j\,|\,i,u) = 1 \ \forall \ (i,u)\in S\times \Act$ be its controlled
transition kernel, that is, the following `controlled Markov property' holds:
$$P(X_{n+1} = j \,|\, X_m,Z_m, m \leq n) \,=\,
 p(j\,|\,X_n,Z_n)\,,\ \ n \geq 0\,.$$
Such $\{Z_n\}$ will be called admissible controls.
We call $\{Z_n\}$ a stationary (randomized) policy if
$$P(Z_n = u \,|\,X_m, m \leq n;\, Z_m, m < n) \,=\, \varphi(u\,|\,X_n)$$ 
for some $\varphi\colon i \in S \mapsto \varphi( \cdot \smid i) \in \PA(\Act)$,
with $\PA(\Act)$ denoting the simplex of probability vectors on $\Act$.
A stationary policy is called
\emph{pure} or \emph{deterministic} if $Z_n = v(X_n)$ for all $n \ge 0$,
for some $v\colon S \mapsto \Act$, equivalently, when $\varphi( \cdot | i ) = \delta_{v(i)}(\cdot) \ \forall i$, i.e., a Dirac measure at $v(i) \ \forall i$.
We let $\Usm$ and $\Upm$ denote the class of all stationary
and pure policies, respectively.
By abuse of  terminology,  stationary policies, resp.\ pure policies, are identified with the map $\varphi$, resp.\ $v$,
in the preceding definition.

The risk-sensitive cost minimization problem we are interested in seeks  to determine
\begin{align}
\Bar\lambda^* \, &\df \, \max_{i\in S}\, \lambda^*_i\,, \label{cost} \\
\lambda^*_i \, &\df \, \inf_{\{Z_m\}} \,\limsup_{n\uparrow\infty}\,
\frac{1}{n}\log \Exp_i\Bigl[\E^{\sum_{m=0}^{n-1}c(X_m, Z_m)}\Bigr]\,, \label{cost-i}
\end{align}
where the infimum is over all admissible controls,
and $\Exp_i[\,\cdot\,]$ denotes the
expectation with $X_0 = i$.
We restrict ourselves to stationary policies.
For a stationary policy $v\in\Usm$, we use the notation
\begin{equation}\label{E-not1}
\begin{aligned}
c_v(i) &\,\df\, \sum_{u\in \Act} c(i,u) v(u\,|\,i)\,,\\[3pt]
p_v(j\,|\,i) &\,\df\, \sum_{u\in \Act} p(j\,|\,i,u)v(u\,|\,i)\,.
\end{aligned}
\end{equation}
Let
$$\lambda_i^v \,\df\, \lim_{n\uparrow\infty}\,\frac{1}{n}\,
\log \Exp_i^v\Bigl[\E^{\sum_{m=0}^{n-1}c_v(X_m)}\Bigr]\,,$$
where $\Exp_i^ v[ \cdot ]$ indicates the  expectation under
the  policy $v\in\Usm$ with $X_0 = i$. Thus 
\begin{equation}\label{E-min}
\Bar\lambda^* \,=\, \min_{v\in\Usm}\,\max_{i\in S}\,\lambda_i^v\,.
\end{equation}

\begin{definition}\label{D2.1}
Let $\cQ$ denote the class of stochastic matrices $q=[q_{ij}]_{i,j\in S}$
such that $$q_{ij}\,=\,0\quad\text{if\ \ }\max_{u\in\Act}\, p(j\,|\, i,u)\,=\,0\,.$$
Also, $\cM_q$  denotes the set of invariant probability
vectors of $q\in\cQ$.

\smallskip
Using the equivalent notation $q(j\,|\,i)=q_{ij}$, we define
\begin{equation}\label{E-runreward}
\Tilde{c}(i,q,u) \,=\, c(i,u) - D\bigl( q(\cdot\,|\,i) \,\|\, p(\cdot\,|\,i,u)\bigr)\,,
\end{equation}
if $q(\cdot\,|\,i)\ll p(\cdot\,|\,i,u)$,
and $\Tilde{c}(i,q,u)=-\infty$, otherwise.
Here, 
\begin{equation*}
D\bigl( q(\cdot\,|\,i) \,\|\, p(\cdot\,|\,i,u)\bigr)
\,=\,
\sum_{j\in S} q(j\,|\,i)\frac{\log q(j\,|\,i)}{\log p(j\,|\,i,u)}
\end{equation*}
denotes the Kullback-Leibler divergence. For $v\in\Usm$, we let $\Tilde{c}_v(i,q)$ be defined analogously to \cref{E-not1},
that is,
\begin{equation*}
\Tilde{c}_v(i,q) \,\df\, \sum_{u\in \Act} \Tilde{c}(i,q,u)\, v(u\,|\,i)\,.
\end{equation*}
\end{definition}

Specializing  \cite[Theorem~3.3]{Ananth} to the above, we have
\begin{equation}\label{variational}
\max_{i\in S}\lambda_i^v \,=\, \max_{q \in\cQ}\,\max_{\pi \in \cM_q}\,
\sum_{i\in S} \pi(i) \Tilde{c}_v(i,q)\,.
\end{equation}
The reason that we can restrict the maximization to the set $\cQ$ is
the following.
Suppose $(\Hat{q},\Hat{\pi})$ is a pair where the maximum in \cref{variational}
is attained. Without loss of generality we may assume that $\Hat\pi$ is
an ergodic measure. It is clear then that we must have
$\Hat{q}(\cdot\,|\,i)\ll p_v(\cdot\,|\,i)$ on the support of $\Hat\pi$,
otherwise $\max_{i\in S}\lambda_i^v=-\infty$, which is not possible.

\Cref{E-min,variational,E-runreward} suggest
an ergodic game for a controlled Markov chain which we describe next.

\smallskip
\begin{definition}\label{D2.2}
The model for the controlled Markov chain $\{\widetilde{X}_n\}$ is 
as follows:
\begin{itemize}
\item The state space is $S$.

\item The action space is $\cQ(i)\times\Act$, for $i\in S$,
where
\begin{equation}\label{E-Ai}
\cQ(i) \,\df\, \bigl\{q_{ij}\,\colon j\in S\,,\, q\in\cQ\bigr\}\,.
\end{equation}

\item The controlled transition probabilities is are dictated by $q\in\cQ$.
Note then that $\cQ$ may be viewed as the set of stationary policies
with action spaces $\{\cQ(i),\,i\in S\}$.
It is clear that in this space there is no difference
between randomized and pure policies.

\item The running \textit{reward} is
$\Tilde{c}(i,q,u)$ defined in \eqref{E-runreward}.
\end{itemize}
\end{definition}

\medskip

With $\{\widetilde{X}_n\}_{n\in\NN_0}$ denoting the chain defined above,
and $\widetilde\Exp^v_i$ the expectation operator under
the policy $v\in\Usm$ with $\widetilde{X}_0=i\in S$, define
\begin{equation}\label{E-hPhi}
\widehat\Phi(q,v)\,\df\,
\max_{i\in S}\, \lim_{N\to\infty}\,
\frac{1}{N}\, \widetilde\Exp^v_i
\Biggl[\sum_{k=0}^{N-1} \Tilde{c}_v(\widetilde{X}_n,q)\Biggr]\,,
\end{equation}
with $q\in \cQ$.
The preceding analysis shows that we seek to maximize
$\widehat\Phi(q,v)$ with respect to $q\in \cQ$ and minimize it
with respect to $v\in\Usm$.
This forms a single controller zero-sum ergodic game
between the agent who chooses $q$ to maximize the long-term average value
of the reward $\Tilde{c}(i,q,u)$
and the agent who chooses $u$ to minimize it.
The reason that it is  a single controller game is
because the decisions of the second player affect only the payoff and not
the transition probability.
This facilitates the application of \cite{Vrieze} to derive equivalent
linear programs, which we do in \cref{S3}.
It is clear that
\begin{equation}\label{E-hPhi2}
\widehat\Phi(q,v) \,=\, \max_{\pi\in\cM_q}\,\sum_{i\in S}\pi(i)\,\Tilde{c}_v(i,q) \,.
\end{equation}
Suppose we can show that
\begin{equation*}
\min_{v\in\Usm}\,\max_{q\in\cQ}\,\widehat\Phi(q,v)
\end{equation*}
is attained at some $v^*\in\Upm$.
Then, in view of \cref{E-min,variational}, and the fact
that
\begin{equation*}
\Tilde{c}_v (i,q) \,=\,
c_v(i) - D\bigl(q(\cdot\,|\,i) \bigm\| p_v( \cdot \,|\, i)\bigr)
\quad\forall\,v\in\Upm\,,
\end{equation*}
we obtain
\begin{equation}\label{E-minmax0}
\Bar\lambda^* \,=\, \min_{v\in\Usm}\,\max_{q\in\cQ}\,\widehat\Phi(q,v)\,.
\end{equation}

In fact, in \cref{S3} we show that the game has a value $\widehat\Phi^*$,
that is,
\begin{equation}\label{E-value}
\widehat\Phi^* \,=\, \adjustlimits\inf_{v\in\Usm\,}\sup_{q\in\cQ}\,\widehat\Phi(q,v) \,=\,
\adjustlimits\sup_{q\in\cQ\,}\inf_{v\in\Usm}\,\widehat\Phi(q,v)\,,
\end{equation}
and there exists $v^*\in\Upm$ and $q^*\in \cQ$ such that
\begin{equation}\label{E-saddle}
\widehat\Phi^* \,=\, \inf_{v\in\Usm}\,\widehat\Phi(q^*, v)
\,=\, \sup_{q\in\cQ}\,\widehat\Phi(q, v^*) \,.
\end{equation}
In other words, the pair $(q^*,v^*)$ is optimal.

\section{Equivalent linear programs}\label{S3}

We now adapt  the key results of \cite{Vrieze} relevant for us.
Since \cite{Vrieze} works with finite state and action spaces and $\cQ$ is not finite,
we first replace $\cQ$ by a finite approximation $\cQ_n$  for  $n \geq 1$, of
transition probability kernels $q( \cdot \,|\, \cdot )$ such that 
for all $i,j \in S$,
$q(j\,|\,i)$ takes values in the set of dyadic rationals of the form $\frac{k}{2^n}$
for some $0 \leq k \leq 2^n$.
Let $A_n(i)$ be the corresponding action spaces defined as in \cref{E-Ai},
but with $\cQ$ replaced by $\cQ_n$.
As noted in \cref{D2.2}, $\cQ_n$ may be viewed as the set of stationary policies
with action spaces $\{A_n(i),\,i\in S\}$.
For $(q,v)\in\cQ_n\times\Usm$, we let
\begin{equation}\label{E-Phi}
\Phi_i(q,v)\,\df\, \lim_{N\to\infty}\, \frac{1}{N}\, \widetilde\Exp^v_i
\Biggl[\sum_{k=0}^{N-1} \Tilde{c}_v(\widetilde{X}_n,q)\Biggr]\,,
\quad i\in S\,.
\end{equation}
We consider the corresponding single controller zero-sum
game analogous to the one described in \cref{S2}.
As we show later, the  single controller zero-sum
game with the objective in \eqref{E-Phi} over
$(q,v)\in \cQ_n\times\Usm$ has the following equivalent linear programming
formulation.

\smallskip
\noindent \textbf{Primal~program \cref{LPn}:} The primal variables are
\begin{equation*}
V\,=\,(V_1,\dotsc,V_s)\in\RR^s\,,\quad \beta\,=\,(\beta_1,\dotsc,\beta_s)\in\RR^s\,,
\end{equation*}
and
$$y \,=\, (y_1,\dotsc,y_s)\colon \Act\to \PA(S)\,,$$
and the linear program is the following:
\begin{equation}\label{LPn}
\begin{aligned}
&\text{Minimize\ } \sum_{i\in S} \beta_i \text{\ subject to:}\\
&\beta_i \,\ge\, \sum_{j\in S} q_{ij}\,\beta_j\, \ \ \ \ \forall \ i \in S,\\
&V_i \,\ge\, \sum_{u\in \Act} \Tilde{c}(i,q,u)\,y_i(u) - \beta_i  +
\sum_{j\in S} q_{ij}V_j  \  \forall \ i \in S.
\end{aligned}\tag{$\mathsf{LP}_n$}
\end{equation}

\smallskip
\noindent \textbf{Dual program \cref{LPn'}:} The dual variables are
$$\bigl(\mu(i,q),\,\nu(i,q)\colon (i,q)\in S\times A_n(i)\bigr)\,,$$
and $w=(w_1,\dotsc,w_s)\in\RR^s$, and the dual linear program is:
\begin{equation}\label{LPn'}
\begin{aligned}
&\text{Maximize\ } \sum_{i\in S} w_i \text{\ subject to:}\\
&\sum_{(i,q)\in S\times A_n(i)}\bigl(\delta_{ij}-\Tilde{p}(j\,|\,i,q)\bigr)\mu(i,q)
\,=\, 0\ \ \forall j\in S, \\
&\sum_{(i,q)\in S\times A_n(i)}\bigl(\delta_{ij}-\Tilde{p}(j\,|\,i,q)\bigr)\nu(i,q) \\
&\mspace{120mu}+ \sum_{q\in A_n(j)}\mu(j,q) \,=\,    1
\quad\forall\,j\in S\,,  \\
&\sum_{q\in A_n(i)}
\Tilde{c}(i,u,q)\mu(i,q) \,\ge\, w_i \quad \forall\,(i,u)\in S\times \Act\,, \\
&\mu(i,q),\, \nu(i,q) \,\ge\, 0 \quad\forall\,(i,q)\in S\times A_n(i)\,. 
\end{aligned}\tag{$\mathsf{LP}'_n$}
\end{equation}
In the above constraints, $\delta_{ij}=1$ if $i=j$, and equals $0$ otherwise.

The programs in \cref{LPn,LPn'} are exactly as given in \cite[Section~2]{Vrieze},
with the notation adapted to the current setting.
Arguing as in \cite[Lemma~2.1]{Vrieze},
we deduce that both linear programs are feasible and have bounded
solutions.
We note that $\Tilde{c}$ is extended-valued here, whereas it is
and real-valued and bounded in \cite{Vrieze}.
Nevertheless, note that $q'\in A_n(i)$ can always be selected so that
$\Tilde{c}(i,q',u)>-\infty$, and this shows that
the solution $\beta$ is bounded.

\smallskip
\begin{definition}
Let $(V^n,\beta^n,y^n), (\mu^n,\nu^n,w^n)$ denote solutions
for \cref{LPn,LPn'}, resp., for each $n\in\NN$.
Define
\begin{equation*}
\overline\alpha^n_i\,\df\,\sum_{q\in A_n(i)}\mu^n(i,q)\,,
\end{equation*}
and
\begin{equation}\label{E-alphan}
\alpha^n_i(q) \,\df\,
\begin{cases}
\frac{\mu^n(i,q)}{\overline\alpha^n_i}&\text{if\ } \
\overline\alpha^n_i\ne0\\[5pt]
\frac{\nu^n(i,q)}{\sum_{q\in A_n(i)}\nu^n(i,q)}&\text{otherwise.}
\end{cases}
\end{equation}
\end{definition}

\medskip
The following lemma follows from the results in \cite{Vrieze},
some of them drawn from \cite{Bewley,Partha,Vrieze0}.

\smallskip
\begin{lemma}\label{L3.1}
The  single controller zero-sum
game with the objective in \eqref{E-Phi} over
$(q,v)\in \cQ_n\times\Usm$ has a value
$$\Phi^{(n)} = \bigl(\Phi^{(n)}_1, \dotsc , \Phi^{(n)}_s\bigr)\in\RR^s\,,$$ that is,
\begin{equation}\label{E-game-n}
\begin{aligned}
\Phi^{(n)}_i &\,=\, \adjustlimits\inf_{v\in\Usm\,}\sup_{q\in\cQ_n}\,\Phi_i(q,v)\\
&\,=\, \adjustlimits\sup_{q\in\cQ_n\,}\inf_{v\in\Usm}\,\Phi_i(q,v)\,,
\end{aligned}
\end{equation}
and the following hold:
\begin{itemize}
\item[(a)]
We have $\beta^n=\Phi^{(n)}$, where $\beta^n$ is the solution of $\cref{LPn}$.
\item[(b)]
A pair of optimal stationary policies $(q^*_n, v^*_n)\in\cQ_n\times\Upm$
exists.
\item[(c)] The inner supremum (resp., infimum) in the left (resp., right) hand side
of \eqref{E-game-n} is attained at a stationary (nonrandomized) policy.
\item[(d)]
For any solution $(V^n,\beta^n,y^n)$ of \cref{LPn},
$y^n$ is an optimal policy for player 2.
In other words,
$v^*_n (\cdot\,|\, i) = y^n_i(\cdot)$ for all $i\in S$.
Moreover $y^n$ can be selected so as to induce a pure Markov policy.
\item[(e)]
For any solution $(\mu^n,\nu^n,w^n)$ of \cref{LPn'}, $q^*_n$ can be
selected as
\begin{equation*}
q^*_n(\cdot\,|\,i) \,=\, \sum_{q\in A_n(i)} q(\cdot\,|\,i)\alpha^n_i(q)\,,
\end{equation*}
with $\alpha^n_i$ as defined in \cref{E-alphan}.
\end{itemize}
\end{lemma}

\begin{proof}
The proof is based on the results in \cite{Vrieze}.
However, the roles of the players
should be interchanged, since it is player 1 that does not influence the transition probabilities
in \cite{Vrieze}.  But if we define the expected average payoff $V$ as
$$V(v,q) \,=\, -\Phi(q,v)\,,$$
then with $v\in\Usm$ the stationary strategies of player~1, and
$q\in\cQ_n$ those of player~2,  the model matches exactly that
of \cite{Vrieze}.

That the game has a value and parts (a) and (b)
then follow from \cite[Theorem~2.15]{Vrieze}.
Part (c) is the statement of \cite[Lemma~1.2]{Vrieze}.
Part (d) then follows by considering the second constraint in \cref{LPn} together
with \cite[Lemma~2.14]{Vrieze}.
Part (e) follows from the definitions (2.4)-(2.10) following the proof of
\cite[Lemma~2.2]{Vrieze}
together with \cite[Lemma~2.9]{Vrieze}.
This completes the proof.
\end{proof}

\subsection{The semi-infinite linear programs}

Letting $n\nearrow\infty$, we obtain a pair of semi-infinite linear programs
with $\cQ_n$ replaced by $\cQ$ in
\eqref{LPn}--\eqref{LPn'}, that is, linear programs with finitely many
variables, but infinitely many constraints.
These are as follows:

\smallskip
\noindent \textbf{Primal~program \cref{LP}:} The primal variables are
as in \cref{LPn},
and the program is the following:
\begin{equation}\label{LP}
\begin{aligned}
&\text{Minimize\ } \sum_{i\in S} \beta_i \text{\ subject to:}\\
&\beta_i \,\ge\, \sum_{j\in S} \Tilde{p}(j\,|\,i,q')\beta_j\,,\\
&V_i \,\ge\, \sum_{u\in \Act} \Tilde{c}(i, q', u)\,y_i(u) - \beta_i
+\sum_{j\in S} \Tilde{p}(j\,|\,i,q')V_j\,,\\
&\mspace{50mu}\quad\forall\,q'\in \cQ(i)\,,\ \forall\,i\in S\,.
\end{aligned}\tag{$\mathsf{LP}$}
\end{equation}

\smallskip

\noindent \textbf{Dual program \cref{LP'}:} The dual variables are
$$\bigl(\mu(i,q),\,\nu(i,q)\colon (i,q)\in S\times \cQ(i)\bigr)\,,$$
and $w=(w_1,\dotsc,w_s)\in\RR^s$, and the dual linear program is:
\begin{equation}\label{LP'}
\begin{aligned}
&\text{Maximize\ } \sum_{i\in S} w_i \text{\ subject to:}\\
&\sum_{(i,q)\in S\times \cQ(i)}\bigl(\delta_{ij}-\Tilde{p}(j\,|\,i,q)\bigr)\mu(i,q)
\,=\, 0\quad\forall\,j\in S\,, \\
&\sum_{(i,q)\in S\times \cQ(i)}\bigl(\delta_{ij}-\Tilde{p}(j\,|\,i,q)\bigr)\nu(i,q) \\
&\mspace{120mu}+ \sum_{q\in A(j)}\mu(j,q) \,=\,    1
\quad\forall\,j\in S\,,  \\
&\sum_{q\in \cQ(i)}
\Tilde{c}(i,u,q)\mu(i,q) \,\ge\, w_i \quad \forall\,(i,u)\in S\times \Act\,, \\
&\mu(i,q),\, \nu(i,q) \,\ge\, 0 \quad\forall\,(i,q)\in S\times \cQ(i)\,. 
\end{aligned}\tag{$\mathsf{LP}'$}
\end{equation}

With an eye on the passage from the approximate linear programs
\cref{LPn,LPn'}
on $\cQ_n$ to the analogous semi-infinite linear programs
\cref{LP,LP'} over $\cQ$,
we need the following two lemmas.

\smallskip
\begin{lemma}\label{L3.2}
The sequence $\{\beta^n\}_{n\in\NN}$ converges monotonically to
some $\widehat\beta\in\RR^s$ in each component.
Moreover, $\widehat\beta$ is the infimum of all feasible values of \cref{LP}.
\end{lemma}

\begin{proof}
Since any solution $(V^n,\beta^n,y^n)$ of \cref{LPn} is feasible for the program
$\mathsf{LP}_{n+1}$,
it is clear that $\beta^n$ is nonincreasing
in $n$ in each component.
It also follows by the definition in \cref{E-runreward} that there
exists a constant $M$ such that
\begin{equation*}
\adjustlimits\min_{(i,u)\in S\times\Act\,}\max_{q\in A_n(i)}\,
\Tilde{c}(i,q,u) \,\ge\, M\,.
\end{equation*}
Thus, since $\beta^n_i$ is clearly bounded below by $M$ for each $i\in S$, and $n\in\NN$,
there exists a limit
\begin{equation*}
\widehat\beta\,\df\, \lim_{n\nearrow\infty}\, \beta^n\,.
\end{equation*}
Now, it is straightforward to show that any feasible solution $\beta$ of
\cref{LP} satisfies $\beta\ge\widehat\beta$.
Indeed, if some $\beta$ with $\beta_i<\widehat\beta_i$ is feasible for \cref{LP},
one can find a $\Tilde\beta$ arbitrarily close to $\beta$ which is
feasible for \cref{LPn} for large enough $n$ by continuity.
This of course contradicts the fact that $\beta^n\ge\widehat\beta$ for all $n\in\NN$,
and completes the proof.
\end{proof}

\medskip

Let $P^n$ denote the transition matrix induced by \cref{LPn'} via
the optimal policy defined in \cref{L3.1}\,(d).
Note that this satisfies $\beta^n = P_n\beta^n$ for all $n\in\NN$ by \cref{LPn}
(see \cite[Lemma~2.9]{Vrieze}).
Recall also that $y^n$ is pure Markov, and can be identified with
$v^*_n$ as asserted in \cref{L3.1}\,(d).
We continue with the following lemma.

\smallskip
\begin{lemma}\label{L3.3}
Any limit point $(\widehat\beta,\widehat{P},\Hat{y})$ of
$(\beta^n,P_n,y^n)$ along a subsequence,
as $n\to\infty$ is feasible for \cref{LP}.
\end{lemma}

\begin{proof}
Let $Q_n$ be defined by
\begin{equation*}
Q_n\,\df\, \lim_{N\to\infty}\,\frac{1}{N} \sum_{k=0}^{N-1} P_n^k\,,
\end{equation*}
and similarly define $\widehat{Q}$ relative to the stochastic matrix $\widehat{P}$.
Also let $\Tilde{c}_n$ denote the running cost under $P_n$ and $y^n$.
It is clear that $\Tilde{c}_n$ converges to some $\Hat{c}$ as $n\to\infty$
along the same subsequence.
Since $(\beta^n,P_n,y^n)$ is optimal for \cref{LPn}, we have (in vector notation)
$\beta^n = Q_n \Tilde{c}_n$,
and $\beta^n = P_n\beta^n$ for all $n\in\NN$ by \cref{LPn}
(see \cite[Lemma~2.9]{Vrieze}).
Thus taking limits as $n\to\infty$, we obtain
\begin{equation*}
\widehat\beta\,=\, \widehat{Q}\,\Hat{c}\,,\quad\text{and\ }
\widehat\beta\,=\, \widehat{P}\widehat\beta\,.
\end{equation*}
Note that $V^n$ can be selected as
\begin{equation*}
V^n \,=\, \bigl(I - P_n + Q_n\bigr)^{-1}
\bigl(I - Q_n\bigr) \Tilde{c}_n\,.
\end{equation*}
Taking limits as $n\to\infty$, it follows that $V^n\to\widehat{V}$
which satisfies
\begin{equation*}
\widehat{V} \,\df\, \bigl(I - \widehat{P} + \widehat{Q}\bigr)^{-1}
\bigl(I - \widehat{Q}\bigr) \Hat{c}\,.
\end{equation*}
Inserting the dependence of $\Tilde{c}_n$ and $P_n$ explicitly in the
notation, the second constraint in \cref{LPn} can be written as
\begin{equation}\label{PL3.3F}
V^n + \beta^n \,\ge\, \Tilde{c}_n(q) + P(q) V^n
\end{equation}
for all $q\in\cQ_n$.  Now, fix some $m\in\NN$ and
$q\in\cQ_m$.
Taking limits in \cref{PL3.3F} as $n\to\infty$, we obtain
\begin{equation}\label{PL3.3G}
\widehat{V} + \widehat\beta \,\ge\, \Hat{c}(q) + P(q) \widehat{V}.
\end{equation}
Since $q\in\cQ_m$ is arbitrary and $\cup_{m\in\NN} \cQ_m$ is dense in $\cQ$,
it follows that \cref{PL3.3G} holds for all $q\in\cQ$.
Hence the second constraint in \cref{LP} is satisfied.
Similarly, starting from
\begin{equation*}
\beta^n \,\ge\, P(q) \beta^n \quad\forall q\in\cQ_n\,,
\end{equation*}
and repeating the same argument, we see that the first constraint in
\cref{LP} is also satisfied.
This completes the proof of the lemma.
\end{proof}

\smallskip
\begin{remark}
It is also possible to start from a solution
$(\mu^n,\nu^n,w^n)$ of the dual program \cref{LPn'}, and then take limits as $n\to\infty$.
Note that $\mu^n$ is a Dirac mass, so convergence to (say) $\widehat\mu$
is interpreted in the weak sense. Same for $\nu^n\to\widehat\nu$.
It is easy to see then that any subsequential limit 
$(\widehat\mu,\widehat\nu,\widehat{w})$ satisfies \cref{LP'} by continuity.
\end{remark}

\smallskip
By \cref{L3.2,L3.3}, the linear programs in
\cref{LP,LP'} are feasible and have bounded solutions.
This allows us to extend \cref{L3.1} as follows.

\smallskip
\begin{theorem}\label{T3.1}
The  single controller zero-sum
game with the objective in \eqref{E-Phi} over
$(q,v)\in \cQ\times\Usm$ has a value
$\Phi^* = \bigl(\Phi^*_1, \dotsc , \Phi^*_s\bigr)\in\RR^s$, that is,
\begin{equation*}
\Phi^*_i \,=\, \adjustlimits\inf_{v\in\Usm\,}\sup_{q\in\cQ}\,\Phi_i(q,v)
\,=\, \adjustlimits\sup_{q\in\cQ\,}\inf_{v\in\Usm}\,\Phi_i(q,v)\,,
\end{equation*}
and the following hold:
\begin{itemize}
\item[(i)]
$\Phi^*=\beta$, the solution to \cref{LP}.
\item[(ii)]
A pair of optimal stationary policies $(q^*, v^*)\in\cQ\times\Upm$
exists.
\item[(iii)] The analogous statements of parts (c)--(e) in \cref{L3.1} hold.
\end{itemize}
\end{theorem}

\smallskip
It is now easy to connect the original game in \cref{E-hPhi} to the game
with the objective in \cref{E-Phi}.
Since the maximum of $\Phi_i(q,v)$ over $i\in S$, is attained in some ergodic
class (communicating class of recurrent states), then in view
of \cref{E-hPhi2}, we have
\begin{equation}\label{E-equal1}
\widehat\Phi(q,v) \,=\, \max_{i\in S}\, \Phi_i(q,v)\qquad
\forall (q,v)\in\cQ\times\Usm\,.
\end{equation}
Thus, by \cref{T3.1}, the game in \cref{E-hPhi} has
the value $\widehat\Phi^*=\max_{i\in S}\, \Phi_i^*$,
and \cref{E-value} holds.
In addition, \cref{E-equal1} implies that the pair
$(q^*, v^*)\in\cQ\times\Upm$ in \cref{T3.1}\,(ii) is optimal
for the game in \cref{E-hPhi}, and thus \cref{E-saddle} holds.

In addition, the fact that $v^*\in\Upm$ as asserted in
\cref{T3.1}\,(ii) implies that \cref{E-minmax0} holds.
So, in summary, the risk-sensitive value $\Bar\lambda^*$ defined in \cref{cost}
satisfies
\begin{equation*}
\begin{aligned}
\Bar\lambda^* &\,=\, \adjustlimits\inf_{v\in\Usm\,}\sup_{q\in\cQ}\,\widehat\Phi(q,v) \,=\,
\adjustlimits\sup_{q\in\cQ\,}\inf_{v\in\Usm}\,\widehat\Phi(q,v)\\
&\,=\, \inf_{v\in\Usm}\,\widehat\Phi(q^*, v)
\,=\, \sup_{q\in\cQ}\,\widehat\Phi(q, v^*) \,.
\end{aligned}
\end{equation*}

%
%

%
%
%
%
%
%
%

\section{Dynamic programming}\label{S4}

It can be seen from the linear program \cref{LP}
that the values $\{\Phi^*_i\colon i\in S\}$ can be calculated
by nested dynamic programming equations (see \cite{Puterman}, pp.\ 442--443).
We simplify the notation and write the stochastic matrix $q$
as $[q_{ij}]$.

We have the following theorem.

\smallskip
\begin{theorem}\label{T4.1}
It holds that
$$\Bar\lambda^* \,=\, \max_{i\in S}\, \lambda^*_i
\,=\, \max_{i\in S}\, \Phi^*_i\,,$$
where $\{\Phi^*_i\}_{i\in S}$ solves, for all $i\in S$,
\begin{align}
\Phi^*_i &\,=\, \max_{q \in \cQ}\,
\sum_{j\in S} q_{ij}\Phi^*_j\,, \label{DP-1}\\
\Phi^*_i + V_i &\,=\, \min_{u\in \Act}\,\max_{q\in B_i}\,
\Biggl[\Tilde{c}(i,q,u)+\sum_{j\in S} q_{ij} V_j  \Biggr]\,, \label{DP-2}
\end{align}
with
$$B_i\,\df\, \Biggl\{q\in\cQ\,\colon
\sum_{j\in S} q_{ij}\Phi^*_j = \Phi^*_i\Biggr\}\,.$$
\end{theorem}

Note that \cref{DP-1,DP-2} simply match the constraints in \cref{LP},
so that existence of a solution to these equations follows from
\cref{T3.1}.
The proof of \cref{T4.1}
again goes through a sequence of finite approximations of $\cA$ so that
the aforementioned results from \cite{Puterman} apply.

Care should be taken when performing
the maximization over $q\in\cQ$ in \eqref{DP-2} explicitly using the
Gibbs variational principle (see Proposition~2.3, \cite{DaiPra}),
since the variables $q$ in \cref{DP-2} are not free 
but depend on the maximization in \cref{DP-1}.
Re-order the solution $\{\Phi^*_i\}$ so that
over a partition $\{\cI_1,\dotsc,\cI_m\}$ of $S$, we have
\begin{equation*}
\Phi^*_i\,=\,\beta^*_\ell\quad\forall\,i\in\cI_\ell
\end{equation*}
and
\begin{equation*}
\beta^*_1\,<\,\beta^*_2\,<\,\dotsb\,<\,\beta^*_m\,.
\end{equation*}
It is clear then that
\begin{equation*}
B_i\,=\, \{q_{ij}\in \cQ\,\colon j\in \cI_1\} \quad\forall\,i\in \cI_1\,,
\end{equation*}
and in general
\begin{equation*}
B_i\,=\, \{q_{ij}\in \cQ\,\colon j\in \cI_k\} \quad\forall\,i\in \cI_k\,.
\end{equation*}
Let
\begin{equation*}
\Hat{p}(j\smid i,u) \,\df\,
\begin{cases}
p(j\smid i,u) &\text{if\ } i,j\in \cI_k\ \text{for\ } k\in\{1,\dotsc,m\}\\
0&\text{otherwise.}
\end{cases}
\end{equation*}
Note that the matrix $[\Hat{p}]$ is block-diagonal.
Thus we can write the maximum in \cref{DP-2} as
\begin{equation*}
q^*(j\,|\,i,u) \,\df\, \frac{\Hat{p}(j\,|\,i,u)\E^{c(i,u) + V_j}}
{\sum_k \Hat{p}(k\,|\,i,u)\E^{c(i,u) + V_k}}\,.
\end{equation*}
Substituting this back into \cref{DP-1,DP-2} along with the change of variables
$\Psi_i = \E^{V_i}$, $\Lambda_i = \E^{\Phi^*_i}$, and $\Lambda^* = \E^{\Bar\lambda^*}$,
we get
\begin{align}
\Lambda^* &\,=\, \max_{i \in S}\,\Lambda_i\,, \label{DP*1} \\
\Lambda_i\Psi_i &\,=\, \min_{u\in \Act}\,
\Biggl(\sum_{j \in S} \Hat{p}(j\,|\,i,u)\E^{c(i,u)}\Psi_j\Biggr), \label{DP*2} \\
\Lambda_i &\,=\, \min_{u \in B^*_i}\,\sum_{j \in S}
\Biggl(\frac{\Hat{p}(j\,|\,i,u)\E^{c(i,u)}\Psi_j}
{\sum_k\Hat{p}(k\,|\,i,u)\E^{c(i,u)}\Psi_k}\Biggr)
\Lambda_j \label{DP*3}
\end{align}
for $i \in S$, where $B^*_i$ is the set of minimizers in \eqref{DP*2}.
As in \cite{Borkar}, the important observation here is the appearance of
a `twisted kernel' for averaging in \eqref{DP*2}\footnote{This also serves as a `correction note' to the derivation of (11)-(12) in \cite{Borkar}. The treatment of dynamic programs in \textit{ibid.} is flawed and should be replaced by the exact counterpart of the above.}

\section{Comments on a counterexample of \texorpdfstring{\cite{CavHern-04}}{}}

We discuss the counterexample in \cite[Example~2.1]{CavHern-04}, which
is for an uncontrolled model.

\smallskip
\begin{example}\label{Ex4.1}
Let
\begin{equation*}
p_{21}=1-\rho,\quad p_{22} =\rho,\quad p_{11} =1,\quad c(2)=1\quad c(1)=0\,,
\end{equation*}
with $\rho\in(0,1)$.
Solving \cref{DP-1,DP-2} for $i=1$, we obtain
\begin{equation*}
\Phi^*_1 =0\,,\quad q_{11} =1\,,\quad V_1=\text{any constant}\,.
\end{equation*}

The equations for $i=2$ are
\begin{align*}
\Phi^*_2 &\,=\, \max_{q \in \cQ}\,
\bigl[ q_{22}\Phi^*_2\bigr]\,, \\
\Phi^*_2 + V_2 &\,=\, \max_{q\in B_2}\,
\Biggl[1-q_{22}\log\frac{q_{22}}{\rho} -
q_{21}\log\frac{q_{21}}{1-\rho}\\
&\mspace{200mu}+ q_{22} V_2 + q_{21} V_1 \Biggr] \,.
\end{align*}
Thus we must have $q_{22}=1$ if $\Phi^*_2\ne0$.
In this case, from \cref{DP-1,DP-2}, we get
\begin{equation*}
\Phi^*_2 = 1+\log \rho\,,\quad q_{22} =1\,,\quad V_2=\text{any constant}\,.
\end{equation*}

If $\log\rho>-1$, then the first hitting time to state
$1$ (from state $2$) does not have an exponential moment,
and $\lambda^*_2= 1+\log \rho$, while of course $\lambda^*_1=0$.

On the other hand if $\log\rho<-1$, then $q_{22}\ne1$, and we get
$\Phi^*_2 = 0$, and $q\equiv q_{22}\in(0,1)$ solves
\begin{equation*}
\log\frac{q}{1-q}-\log\frac{\rho}{1-\rho}+ \frac{1}{1-q}\Bigl((1-2q)V_1
- B(q) \Bigr)\,=\,0\,,
\end{equation*}
with
\begin{equation*}
B(q)\,\df\, 1-q\log\frac{q}{\rho} - (1-q) \frac{1-q}{1-\rho}\,.
\end{equation*}
Also $V_2 = \frac{q V_1 + B(q)}{1-q}$.

Thus, in either case, $\Bar\lambda^* = \max\,\{\Phi^*_1,\Phi^*_2\}$.
\end{example}

\smallskip
However, as noted in \cite[Example~2.1]{CavHern-04},
the multiplicative Poisson equation does not have a solution when $\log\rho>-1$,
because there is no pair of numbers $(h_1,h_2)$ that even solves the inequality
\begin{equation*}
\begin{aligned}
\E\rho \E^{h_2}\,=\,\E^{\lambda^*_2} \E^{h_2}
&\,\ge\, \E^{c(2)} \Bigl[ p_{22} \E^{h_2} + p_{21} \E^{h_1}\Bigr]\\
&\,=\, \E\Bigl[ \rho \E^{h_2} + (1-\rho) \E^{h_1}\Bigr]\,.
\end{aligned}
\end{equation*}

\smallskip
We compare \cref{T4.1} with the results in \cite{CavHern-05}.
As shown in \cite[Theorem~3.5]{CavHern-05}, under a Doeblin hypothesis,
it holds that
\begin{equation}\label{E-inf}
\lambda^*_i\,=\,\inf_{g\in\sG}\, g(i)\,,
\end{equation}
where $\sG$ is the class of functions satisfying
\begin{equation*}
g(i) \,=\, \min_{u\in\Act}\,\Bigl(\max\,\{g(j)\,\colon p(j\smid i,u)>0\}\Bigr)\,,
\end{equation*}
and
\begin{equation*}
\E^{g(i) + h_i} \,\ge\,
\min_{u\in B_g(i)}\, \Biggl[ \E^{c(i,u)} \sum_{j\in S} p(j\smid i,u) \E^{h_j}\Biggr]
\,,
\end{equation*}
where $h=(h_1,\dotsc,h_{s})\in\RR^{s}$ is a vector possibly depending
on $g$, and
\begin{equation*}
B_g(i) \,\df\,\bigl\{ u\in\Act\,\colon g(i)=\max\,\{g(j)\,\colon p(j\smid i,u)>0\}\bigr\}\,.
\end{equation*}
It is important to note that the infimum in \cref{E-inf} might not be realized
in $\sG$.  This is what \cref{Ex4.1} shows in the case $\log\rho>-1$.

\section{Future Directions}

One interesting problem that still remains is to show optimality of stationary or pure policies under very general conditions that do not require irreducibility. Yet another interesting direction is an extension of this paradigm to general state spaces and to continuous time risk-sensitive control.

\vfill\eject

\section*{Acknowledgement}
The work of AA was supported in part by
the National Science Foundation through grant DMS-1715210, and in part
the Army Research Office through grant W911NF-17-1-001,
while the work of VB was supported in part by  J.\ C.\ Bose and S.\ S.\ Bhatnagar Fellowships from the Government of India.

\bigskip


\begin{thebibliography}{99}

\bibitem{Ananth} Anantharam, V., and Borkar, V.\ S.\ (2017)
``A variational formula for risk-sensitive reward'',
\textit{SIAM Journal on Control and Optimization}, 55(2), 961--988.


\bibitem{Bewley} Bewley, T., and Kohlberg, E.\ (1978) ``On stochastic games with
stationary optimal strategies'', \textit{Math.\ Op.\ Research} 3, 104--125.

\bibitem{Borkar} Borkar, V.\ S.\ (2017) ``Linear and dynamic programming approaches
to degenerate risk-sensitive reward processes'',
\textit{Proc.\  IEEE 56th Annual Conference on Decision and Control (CDC)}, 3714--3718. 

\bibitem{DaiPra} Dai Pra, P.; Meneghini, L.\ and Runggaldier, W.\ J.\ (1996)
``Connections between stochastic control and dynamic games'',
\textit{Mathematics of Control, Signals and Systems}  9(4), 303--326.

\bibitem{Flem1} Fleming, W.\ H., and Hern\'andez-Hern\'{a}ndez, D.\ (1996)
``Risk-sensitive control for finite state machines on infinite horizon I'',
\textit{SIAM J.\ Control and Optim.} 35(5), 1790--1810.

\bibitem{Flem2} Fleming, W.\ H., and Hern\'andez-Hern\'{a}ndez, D.\ (1996)
``Risk-sensitive control for finite state machines on infinite horizon II'',
\textit{SIAM J.\ Control and Optim.} 37(4), 1048--1069.

\bibitem{CavHern-04} Cavazos-Cadena, R. and Hern\'andez-Hern\'andez, D. (2004)
``A characterization of exponential functionals in finite  Markov chains'',
\textit{Math. Meth. Oper. Res.} 60(3), 399-414.

\bibitem{CavHern-05} Cavazos-Cadena, R. and Hern\'andez-Hern\'andez, D. (2005)
``A characterization of the optimal risk-sensitive average cost in finite
controlled Markov chains'', \textit{Ann. Appl. Probab.} 15(1A), 175--212.


\bibitem{Partha} Parthasarathy, T., and Raghavan, T.\ E.\ S.\ (1981)
``An order field property for stochastic games when one player controls the transitions'',
\textit{J.\ Opt.\ Theory and Appl.} 33(3), 375--392.

\bibitem{Puterman} Puterman, M.\  L.\ (1994)
\textit{Markov Decision Processes: Discrete Stochastic Dynamic Programming},
John Wiley and Sons, Hoboken, NJ.


\bibitem{Vrieze0} Vrieze, O.\ J.\ (1979)
``Characterization of optimal stationary strategies in undiscounted stochastic games'', 
Report BW102/79, Stichting Math. Centrum, Amsterdam.


\bibitem{Vrieze} Vrieze, O.\ J.\ (1981)
``Linear programming and undiscounted stochastic games in which one player controls
the transitions'', \textit{OR Spektrum} 3, 29--35.

\end{thebibliography}
\end{document}